\documentclass[reqno,12pt]{amsart}
\usepackage{txfonts}
\usepackage{mathrsfs}
\usepackage{amsmath}
\usepackage{amssymb}

\usepackage{amsmath,amssymb,amsthm}
\usepackage{verbatim}
\usepackage{latexsym, bm}

\title[COMPLEX PLATEAU PROBLEM]{KOHN--ROSSI COHOMOLOGY AND ITS APPLICATION TO THE COMPLEX PLATEAU PROBLEM, III}

\author{}
\author[Rong Du]{Rong Du$^{\dag}$}
\author[Stephen Yau]{Stephen Yau$^{\ast}$}
\address{Department of Mathematics\\
East China Normal University\\
No. 500, Dongchuan Road\\
Shanghai, 200241, P. R. China\\}

\email{rdu@math.ecnu.edu.cn}

\address{Department of mathematical sciences\\
Tsinghua University\\
Beijing, 100084, P.R.China\\}

\email{yau@uic.edu}
\thanks{$^{\dag}$ Research supported by National Natural Science Foundation of China and Innovation Foundation of East China Normal University.}
\thanks{$^{\ast}$ Research partially supported by NSF and Department of Mathematical Sciences, Tsinghua University, Beijing, P.R.China.}

\theoremstyle{definition}
\newtheorem{theorem}[subsection]{Theorem}
\newtheorem{lemma}[subsection]{Lemma}
\newtheorem{definition}[subsection]{Definition}
\newtheorem{proposition}[subsection]{Proposition}

\newtheorem{corollary}[subsection]{Corollary}

\newfont{\drnew}{wncyr10}
\let\tilde=\widetilde

\allowdisplaybreaks

\def\dashfill{\leaders\hbox{\hbox to 3.25pt{\hrulefill}\hspace*{2pt}\hbox to 3.25pt{\hrulefill}}\hfill}
\newcommand{\CITE}[1]{{[#1]}}
\let\cite=\CITE

\begin{document}

\begin{abstract}
Let $X$ be a compact connected strongly pseudoconvex $CR$ manifold
of real dimension $2n-1$ in $\mathbb{C}^{N}$. It has been an
interesting question to find an intrinsic smoothness criteria for
the complex Plateau problem. For $n\ge 3$ and $N=n+1$, Yau found a
necessary and sufficient condition for the interior regularity of
the Harvey--Lawson solution to the complex Plateau problem by means
of Kohn--Rossi cohomology groups on $X$ in 1981. For $n=2$ and $N\ge
n+1$, the problem has been open for over 30 years. In this paper we
introduce a new CR invariant $g^{(1,1)}(X)$ of $X$. The vanishing of
this invariant will give the interior regularity of the
Harvey--Lawson solution up to normalization. In the case $n=2$ and
$N=3$, the vanishing of this invariant is enough to give the
interior regularity.
\end{abstract}

\maketitle

{\small{Dedicated to Professor Blaine Lawson on the occasion of his
$68^{\text{th}}$ Birthday.}}

\vspace{1cm}
\section{\textbf{Introduction}}
One of the natural fundamental questions of complex geometry is to
study the boundaries of complex varieties. For example, the famous
classical complex Plateau problem asks which odd-dimensional real
sub-manifolds of $\mathbb{C}^N$ are boundaries of complex
sub-manifolds in $\mathbb{C}^N$. In their beautiful seminal paper,
Harvey and Lawson [Ha-La] proved that for any compact connected $CR$
manifold $X$ of real dimension $2n-1$, $n\ge 2$, in $\mathbb{C}^N$,
there is a unique complex variety $V$ in $\mathbb{C}^N$ such that
the boundary of $V$ is $X$. In fact, Harvey and Lawson proved the
following theorem.

\vspace{.5cm} \textbf{Theorem} (Harvey--Lawson [Ha-La1, Ha-La2]) Let
$X$ be an embeddable strongly pseudoconvex $CR$ manifold. Then $X$
can be $CR$ embedded in some $\mathbb{C}^{\tilde{N}}$ and $X$ bounds
a Stein variety $V \subseteq \mathbb{C}^{\tilde{N}}$ with at most
isolated singularities.

The above theorem is one of the deepest theorems in complex
geometry. It relates the theory of strongly pseudoconvex $CR$ manifolds
on the one hand and the theory of isolated normal singularities on
the other hand.

The next fundamental question is to determine when $X$ is a boundary
of a complex sub-manifold in $\mathbb{C}^N$, i.e., when $V$ is
smooth.  In 1981, Yau [Ya] solved this problem for the case $n\ge 3$
by calculation of Kohn--Rossi cohomology groups $H^{p, q}_{K R}(X)$.
More precisely, suppose $X$ is a compact connected strongly
pseudoconvex $CR$ manifold of real dimension $2n-1$, $n\ge 3$, in
the boundary of a bounded strongly pseudoconvex domain $D$ in
$\mathbb{C}^{n+1}$. Then $X$ is a boundary of the complex
sub-manifold $V\subset D-X$ if and only if Kohn--Rossi cohomology
groups $H^{p, q}_{K R}(X)$ are zeros for $1\le q\le n-2$ (see
Theorem \ref{Yau Pla}).

Kohn--Rossi cohomology introduced by Kohn and Rossi [Ko-Ro] in 1965
is a fundamental invariant of $CR$ manifold. In the recent work of
Huang, Luk,  and Yau [H-L-Y], it was shown that the Kohn--Rossi
cohomology plays an important role in the simultaneous $CR$
embedding of a family of strongly pseudoconvex $CR$ manifolds of
dimension at least 5.

For $n=2$, i.e., $X$ is a 3-dimensional $CR$ manifold, the intrinsic
smoothness criteria for the complex Plateau problem remains unsolved
for over a quarter of a century even for the hypersurface case. The
main difficulty is that the Kohn--Rossi cohomology groups are
infinite-dimensional in this case. Let $V$ be a complex variety with
$X$ as its boundary. Then the singularities of $V$ are surface
singularities. In [Lu-Ya2], the holomorphic De Rham cohomology,
which is derived form Kohn--Rossi cohomology, is considered to
determine what kind of singularities can happen in $V$ . In fact, in
[Ta], Tanaka introduced a spectral sequence $E^{p, q}_r(X)$ with
$E^{p, q}_1(X)$ being the Kohn--Rossi cohomology group and $E^{k,
0}_2(X)$ being the holomorphic De Rham cohomology denoted by
$H^k_h(X)$. So consideration of De Rham cohomology is natural in the
case of $n=2$. Motivated by the deep work of Siu [Si], Luk and Yau
introduced the Siu complex and s-invariant (see Definition \ref{s
inv}, below) for isolated singularity $(V, 0)$ and proved a theorem
in [Lu-Ya2] that if $(V, 0)$ is a Gorenstein surface singularity
with vanishing $s$-invariant, then $(V, 0)$ is a quasihomogeneous
singularity whose link is rational homology sphere. In [Lu-Ya2],
they proved that if $X$ is a strongly pseudoconvex compact
Calabi--Yau $CR$ manifold of dimension 3 contained in the boundary
of a strongly pseudoconvex bounded domain $D$ in $\mathbb{C}^N$ and
the holomorphic De Rham cohomology $H^2_h(X)$ vanishes, then $X$ is
a boundary of a complex variety $V$ in $D$ with boundary regularity
and $V$ has only isolated singularities in the interior and the
normalizations of these singularities are Gorenstein surface
singularities with vanishing s-invariant (see Theorem \ref{th
h2=0}). As a corollary of this theorem, they get that if $N=3$, the
variety $V$ bounded by $X$ has only isolated quasi-homogeneous
singularities such that the dual graphs of the exceptional sets in
the resolution are star shaped and all the curves are rational (see
Corollary \ref{h2=0}). Even though one cannot judge when $X$ is a
boundary of a complex manifold with the vanishing of $H^2_h(X)$, it
is a fundamental step toward the solution of the regularity of the
complex Plateau problem. In this paper, we introduce a new $CR$
invariant $g^{(1,1)}(X)$ which has independent interest besides its
application to the complex Plateau problem. Roughly speaking, our
new invariant $g^{(1,1)}(X)$ is the number of independent
holomorphic 2-forms on $X$ which cannot be written as a linear
combination of those elements of the form holomorphic 1-form wedge
with holomorphic 1-form on $X$ . This new invariant will allow us to
solve the intrinsic smoothness criteria up to normalization for the
classical complex Plateau problem for $n=2$.

\vspace{.5cm}\textbf{Theorem A}  \emph{Let $X$ be a strongly
pseudoconvex compact Calabi--Yau $CR$ manifold of dimension $3$.
Suppose that $X$ is contained in the boundary of a strongly
pseudoconvex bounded domain $D$ in $\mathbb{C}^N$ with $H^2_h(X)=0$.
Then $X$ is a boundary of the complex sub-manifold up to
normalization $V\subset D-X$ with boundary regularity if and only if
$g^{(1,1)}(X)=0$.}

\vspace{.5cm}Thus, the interior regularity of the complex Plateau
problem is solved up to normalization. As a corollary of Theorem A,
we have solved the interior regularity of the complex Plateau
problem in case $X$ is of real codimension 3 in $\mathbb{C}^3$.

\vspace{.5cm}\textbf{Theorem B} \emph{Let $X$ be a strongly
pseudoconvex compact $CR$ manifold of dimension $3$. Suppose that
$X$ is contained in the boundary of a strongly pseudo-convex bounded
domain $D$ in $\mathbb{C}^3$ with $H^2_h(X)=0$. Then $X$ is a
boundary of the complex sub-manifold $V\subset D-X$ if and only if
$g^{(1,1)}(X)=0$.}

\vspace{.5cm} In Section 2, we shall recall the definition of
holomorphic De Rham cohomology for a $CR$ manifold. In Section 3,
after recalling several local invariants of isolated singularity, we
introduce some new invariants of singularities and new $CR$
invariants for $CR$ manifolds. In Section 4, we prove the main
theorem of this paper.

Finally, we would like to thank Professor Lawrence Ein and Professor
Anatoly Libgober for helpful discussions.

\section{\textbf{Preliminaries}}

Kohn--Rossi cohomology was first introduced by Kohn--Rossi.
Following Tanaka [Ta], we reformulate the definition in a way
independent of the interior manifold.

\begin{definition}
\emph{Let $X$ be a connected orientable manifold of real dimension
$2n-1$. A $CR$ structure on $X$ is an $(n-1)$-dimensional sub-bundle
$S$ of $\mathbb{C}T(X)$ (complexified tangent bundle) such that:
\begin{enumerate}
\item[1 .]
$S\bigcap \bar{S}=\{0\}$.
\item[2 .]
If $L$, $L'$ are local sections of $S$, then so is $[L, L']$.
\end{enumerate}}
\end{definition}

Such a manifold with a $CR$ structure is called a $CR$ manifold.
There is a unique sub-bundle $\mathcal{H}$ of $T(X)$ such that
$\mathbb{C}\mathcal{H}=S\bigoplus \bar{S}$. Furthermore, there is a
unique homomorphism $J$ : $\mathcal{H}\longrightarrow\mathcal{H}$
such that $J^2=-1$ and $S=\{v-iJv : v\in\mathcal{H}\}$. The pair
$(\mathcal{H}, J)$ is called the real expression of the $CR$
structure.

Let $X$ be a $CR$ manifold with structure $S$. For a complex valued
$C^\infty$ function $u$ defined on $X$, the section
$\bar{\partial}_bu\in \Gamma(\bar{S}^*)$ is defined by
$$\bar{\partial}_bu(\bar{L})=\bar{L}(u), L\in S.$$
The differential operator $\bar{\partial}_b$ is called the
(tangential) Cauchy--Riemann operator, and a solution $u$ of the
equation $\bar{\partial}_bu=0$ is called a holomorphic function.

\begin{definition}\label{hvb on X}
\emph{A complex vector bundle E over X is said to be holomorphic if
there is a differential operator
\[
\bar{\partial}_E: \Gamma(E)\longrightarrow \Gamma(E\otimes
\bar{S}^{*})
\]
satisfying the following conditions:
\begin{enumerate}
\item[1 .]
$\bar{\partial}_E(fu)(\bar{L}_1)=(\bar{\partial}_bf)(\bar{L}_1)u+f(\bar{\partial}_Eu)(\bar{L}_1)
=(\bar{L}_1f)u+f(\bar{\partial}_Eu)(\bar{L}_1)$.
\item[2. ]
$(\bar{\partial}_Eu)[\bar{L}_1,
\bar{L}_2]=\bar{\partial}_E(\bar{\partial}_Eu(\bar{L}_2))(\bar{L}_1)-\bar{\partial}_E(\bar{\partial}_Eu(\bar{L}_1))(\bar{L}_2)$,
where $u\in \Gamma(E)$, $f\in C^\infty(X)$, and $L_1$, $L_2 \in
\Gamma(S)$.
\end{enumerate}}
\end{definition}

The operator $\bar{\partial}_E$ is called the Cauchy--Riemann
operator and a solution $u$ of the equation $\bar{\partial}_Eu=0$ is
called a holomorphic cross section.

A basic holomorphic vector bundle over a CR manifold $X$ is the
vector bundle $\widehat{T}(X)=\mathbb{C}T(X)/\bar{S}$. The
corresponding operator
$\bar{\partial}=\bar{\partial}_{\widehat{T}(X)}$ is defined as
follows. Let $p$ be the projection from $\mathbb{C}T(X)$ to
$\widehat{T}(X)$. Take any $u \in \Gamma(\widehat{T}(X))$ and
express it as $u=p(Z)$, $Z \in \Gamma(\mathbb{C}T(X))$. For any
$L\in \Gamma(S)$, define a cross section
$(\bar{\partial}u)(\bar{L})$ of $\widehat{T}(X)$ by
$(\bar{\partial}u)(\bar{L})= p([\bar{L},Z])$. One can show that
$(\bar{\partial}u)(\bar{L})$ does not depend on the choice of $Z$
and that $\bar{\partial}u$ gives a cross section of
$\widehat{T}(X)\otimes \bar{S}^{*}$. Furthermore, one can show that
the operator $u\longmapsto \bar{\partial}u$ satisfies $(1)$ and
$(2)$ of Definition \ref{hvb on X}, using the Jacobi identity in the
Lie algebra $\Gamma(\mathbb{C}T(X))$. The resulting holomorphic
vector bundle $\widehat{T}(X)$ is called the holomorphic tangent
bundle of $X$.

If $X$ is a real hypersurface in a complex manifold $M$, we may
identify $\widehat{T}(X)$ with the holomorphic vector bundle of all
$(1,0)$ tangent vectors to $M$ and $\widehat{T}(X)$ with the
restriction of $\widehat{T}(M)$ to $X$. In fact, since the structure
$S$ of $X$ is the bundle of all $(1,0)$ tangent vectors to $X$, the
inclusion map $\mathbb{C}T(X)\longrightarrow\mathbb{C}T(M)$ induces
a natural map
$\widehat{T}(X)\xrightarrow[\phantom{ttttt}]{\phi}\widehat{T}(M)|_X$
which is a bundle isomorphism satisfying
$\bar{\partial}(\phi(u))(\bar{L})=\phi(\bar{\partial}u(\bar{L}))$,
$u\in\Gamma(\widehat{T}(X))$, $L\in S$.

For a holomorphic vector bundle $E$ over $X$, set
\[
C^q(X, E)=E\otimes \wedge^q\bar{S}^{*}, \mathscr{C}^q(X,
E)=\Gamma(C^q(X, E))
\]
and define a differential operator
\[
\bar{\partial}^q_E : \mathscr{C}^q(X, E)\longrightarrow
\mathscr{C}^{q+1}(X, E)
\]
by
\[
(\bar{\partial}^q_E\phi)(\bar{L}_1, \dots , \bar{L}_{q+1})
=\sum_i(-1)^{i+1}\bar{\partial}_E(\phi(\bar{L}_1, \dots ,
\widehat{\bar{L}_i}, \dots , \bar{L}_{q+1}))(\bar{L}_i)
\]
\[
+\sum_{i<j}(-1)^{i+j}\phi([\bar{L}_i, \bar{L}_j], \bar{L}_1, \dots ,
\widehat{\bar{L}_i}, \dots , \bar{L}_{q+1})
\]
for all $\phi\in\mathscr{C}^q(X, E)$ and $L_1, \dots ,
L_{q+1}\in\Gamma(S)$. One shows by standard arguments that
$\bar{\partial}^q_E\phi$ gives an element of $\mathscr{C}^{q+1}(X,
E)$ and that $\bar{\partial}^{q+1}_E\bar{\partial}^q_E=0$. The
cohomology groups of the resulting complex $\{\mathscr{C}^{q}(X, E),
\bar{\partial}^{q}_E \}$ is denoted by $H^q(X, E)$.

Let $\{\mathscr{A}^k(X), d\}$ be the De Rham complex of $X$ with
complex coefficients, and let $H^k(X)$ be the De Rham cohomology
groups. There is a natural filtration of the De Rham complex, as
follows. For any integer $p$ and $k$, put
$A^k(X)=\wedge^k(\mathbb{C}T(X)^{*})$ and denote by $F^p(A^k(X))$
the sub-bundle of $A^k(X)$ consisting of all $\phi\in A^k(X)$ which
satisfy the equality
\[
\phi(Y_1, \dots, Y_{p-1}, \bar{Z}_1, \dots, \bar{Z}_{k-p+1})=0
\]
for all $Y_1, \dots, Y_{p-1} \in \mathbb{C}T(X)_0$ and $Z_1, \dots,
Z_{k-p+1}\in S_0$, $0$ being the origin of $\phi$. Then
\[
A^k(X)=F^0(A^k(X))\supset F^1(A^k(X)) \supset \cdots
\]
\[
\supset F^k(A^k(X)) \supset F^{k+1}(A^k(X))=0.
\]
Setting $F^p(\mathscr{A}^k(X))=\Gamma(F^p(A^k(X)))$, we have
\[
\mathscr{A}^k(X)=F^0(\mathscr{A}^k(X))\supset
F^1(\mathscr{A}^k(X))\supset \cdots
\]
\[
\supset F^k(\mathscr{A}^k(X))\supset F^{k+1}(\mathscr{A}^k(X))=0.
\]
Since clearly $dF^p(\mathscr{A}^k(X))\subseteq
F^p(\mathscr{A}^{k+1}(X))$, the collection
$\{F^p(\mathscr{A}^k(X))\}$ gives a filtration of the De Rham
complex.

We denote by $H^{p, q}_{KR}(X)$ the groups $E^{p, q}_1(X)$ of the
spectral sequence $\{E^{p, q}_r(X)\}$ associated with the filtration
$\{F^p(\mathscr{A}^k(X))\}$. We call $H^{p, q}_{KR}(X)$ the
Kohn--Rossi cohomology group of type $(p, q)$. More explicitly, let
\[
A^{p, q}(X)=F^p(A^{p+q}(X)), \mathscr{A}^{p, q}(X)=\Gamma(A^{p,
q}(X)),
\]
\[
C^{p, q}(X)=A^{p, q}(X)/A^{p+1, q-1}(X),  \mathscr{C}^{p,
q}(X)=\Gamma(C^{p, q}(X)).
\]
Since $d : \mathscr{A}^{p, q}(X) \longrightarrow \mathscr{A}^{p,
q+1}(X)$ maps $\mathscr{A}^{p+1, q-1}(X)$ into $\mathscr{A}^{p+1,
q}(X)$, it induces an operator $d'': \mathscr{C}^{p,
q}(X)\longrightarrow \mathscr{C}^{p, q+1}(X)$. $H^{p, q}_{KR}(X)$
are then the cohomology groups of the complex $\{\mathscr{C}^{p,
q}(X), d''\}$.

Alternatively, $H^{p, q}_{KR}(X)$ may be described in terms of the
vector bundle $E^p=\wedge^p(\widehat{T}(X)^{*})$. If for $\phi\in
\Gamma(E^p)$, $u_1,\dots, u_p\in \Gamma(\widehat{T}(X))$, $Y\in S$,
we define $(\bar{\partial}_{E^p}\phi)(\bar{Y})=\bar{Y}\phi$ by
\[
\bar{Y}\phi(u_1,\dots, u_p)=\bar{Y}(\phi(u_1,\dots,
u_p))+\sum_i(-1)^i\phi(\bar{Y}u_i, u_1,\dots, \widehat{u_i},\dots,
u_p)
\]
where $\bar{Y}u_i=(\bar{\partial}_{\widehat{T}(X)}u_i)(\bar{Y})$,
then we easily verify that $E^p$ with $\bar{\partial}_{E^p}$ is a
holomorphic vector bundle. Tanaka [Ta] proves that $C^{p,q}(X)$ may
be identified with $C^q(X, E^p)$ in a natural manner such that
\[
d''\phi=(-1)^p\bar{\partial}_{E^p}\phi, \phi \in \mathscr{C}^{p,
q}(X).
\]
Thus, $H^{p, q}_{KR}(X)$ may be identified with $H^q(X, E^p)$.

We denote by $H^k_h(X)$ the groups $E^{k,0}_2(X)$ of the spectral
sequence $\{E^{p,q}_r(X)\}$ associated with the filtration
$\{F^p(\mathscr{A}^k(X))\}$. We call $H^k_h(X)$ the holomorphic De
Rham cohomology groups. The groups $H^k_h(X)$ are the cohomology
groups of the complex $\{\mathscr{S}^k(X), d\}$, where we put
$\mathscr{S}^k(X)=E^{k,0}_1(X)$ and $d=d_1 :
E^{k,0}_1\longrightarrow E^{k+1,0}_1$. Recall that
$\mathscr{S}^k(X)$ is the kernel of the following mapping:
\[
d_0: E^{k, 0}_0=F^k\mathscr{A}^k=\mathscr{A}^{k,0}(X)\hspace{4cm}
\]
\[
\hspace{3cm}\rightarrow
E^{k,1}_0=F^k\mathscr{A}^{k+1}/F^{k+1}\mathscr{A}^{k+1}=\mathscr{A}^{k,1}(X)/\mathscr{A}^{k+1,0}.
\]
Note that $\mathscr{S}$ may be characterized as the space of
holomorphic $k$-forms, namely holomorphic cross sections of $E^k$.
Thus, the complex $\{\mathscr{S}^k(X), d\}$ (respectively, the
groups $H^k_h(X)$) will be called the holomorphic De Rham complex
(respectively, the holomorphic De Rham cohomology groups).

\begin{definition}
\emph{Let $L_1,\dots, L_{n-1}$ be a local frame of the $CR$
structure $S$ on $X$ so that $\bar{L}_1,\dots,\bar{L}_{n-1}$ is a
local frame of $\bar{S}$. Since $S\oplus  \bar{S}$ has complex
codimension $1$ in $\mathbb{C}T(X)$, we may choose a local section N
of $\mathbb{C}T(X)$ such that $L_1,\dots, L_{n-1},
\bar{L}_1,\dots,\bar{L}_{n-1}$, $N$ span $\mathbb{C}T(X)$. We may
assume that $N$ is purely imaginary. Then the matrix $(c_{ij})$
defined by
\[
[L_i,
\bar{L}_j]=\sum_ka^k_{i,j}L_k+\sum_kb^k_{i,j}\bar{L}_k+c_{i,j}N
\]
is Hermitian, and it is called the Levi form of $X$.}
\end{definition}
\begin{proposition}
\emph{The number of nonzero eigenvalues and the absolute value of
the signature of $(c_{ij})$ at each point are independent of the
choice of $L_1,\dots, L_{n-1}, N$.}
\end{proposition}
\begin{definition}
\emph{$X$ is said to be strongly pseudoconvex if the Levi form is
positive definite at each point of $X$.}
\end{definition}

\begin{definition}
\emph{Let $X$ be a CR manifold of real dimension $2n-1$. $X$ is said
to be Calabi--Yau if there exists a nowhere vanishing holomorphic
section in $\Gamma(\wedge^n\widehat{T}(X)^*)$, where
$\widehat{T}(X)$ is the holomorphic tangent bundle of $X$.}
\end{definition}

\textbf{Remark}:
\begin{enumerate}
\item[1 .]
Let $X$ be a $CR$ manifold of real dimension $2n-1$ in
$\mathbb{C}^n$. Then $X$ is a Calabi--Yau $CR$ manifold.
\item[2 .]
Let $X$ be a strongly pseudoconvex $CR$ manifold of real dimension
$2n-1$ contained in the boundary of bounded strongly pseudoconvex
domain in $\mathbb{C}^{n+1}$. Then $X $ is a Calabi--Yau $CR$
manifold.
\end{enumerate}

\section{\textbf{Invariants of singularities and $CR$-invariants}}
Let V be a $n$-dimensional complex analytic subvariety in
$\mathbb{C}^N$ with only isolated singularities. In [Ya2], Yau
considered four kinds of sheaves of germs of holomorphic $p$-forms:
\begin{enumerate}
\item[1 .]
$\bar{\Omega}^p_V:=\pi_*\Omega^p_M$, where $\pi: M\longrightarrow V$
is a resolution of singularities of $V$.
\item[2 .]
$\bar{\bar{\Omega}}^p_V:=\theta_*\Omega^p_{V\backslash V_{sing}}$
where $\theta : V\backslash V_{sing}\longrightarrow V$ is the
inclusion map and $V_{sing}$ is the singular set of $V$.
\item[3 .]
$\Omega^p_V:=\Omega_{\mathbb{C}^N}^p/\mathscr{K}^p$, where
$\mathscr{K}^p=\{f\alpha+dg\wedge\beta :
\alpha\in\Omega_{\mathbb{C}^N}^p; \beta\in
\Omega_{\mathbb{C}^N}^{p-1}; f, g\in\mathscr{I}\}$ and $\mathscr{I}$
is the ideal sheaf of $V$ in $\mathbb{C}^N$.
\item[4 .]
$\widetilde{\Omega}^p_V:=\Omega_{\mathbb{C}^N}^p/\mathscr{H}^p$,
where $\mathscr{H}^p=\{\omega\in\Omega_{\mathbb{C}^N}^p:
\omega|_{V\backslash V_{sing}}=0\}$.
\end{enumerate}

Clearly ${\Omega}^p_V$, $\widetilde{\Omega}^p_V$ are coherent.
$\bar{\Omega}^p_V$ is a coherent sheaf because $\pi$ is a proper
map. $\bar{\bar{\Omega}}^p_V$ is also a coherent sheaf by a theorem
of Siu (see Theorem A of [Si]). If $V$ is a normal variety, the
dualizing sheaf $\omega_V$ of Grothendieck is actually the sheaf
$\bar{\bar{\Omega}}^n_V$.

\begin{definition}
\emph{The Siu complex is a complex of coherent sheaves $J^{\bullet}$
supported on the singular points of $V$ which is defined by the
following exact sequence:
\begin{equation}\label{J}
0\longrightarrow\bar{\Omega}^{\bullet}\longrightarrow\bar{\bar{\Omega}}^{\bullet}\longrightarrow
J^{\bullet}\longrightarrow 0.
\end{equation}}
\end{definition}

\begin{definition}\label{s inv}
\emph{Let $V$ be a $n$-dimensional Stein space with $0$ as its only
singular point. Let $\pi: (M, A)\rightarrow (V, 0)$ be a resolution
of the singularity with $A$ as exceptional set. The geometric genus
$p_g$ and the irregularity $q$ of the singularity are defined as
follows (see [Ya2, St-St]):
\begin{equation}
p_g:= dim \Gamma(M\backslash A, \Omega^n)/\Gamma(M, \Omega^n),
\end{equation}
\begin{equation}
q:= dim \Gamma(M\backslash A, \Omega^{n-1})/\Gamma(M, \Omega^{n-1}),
\end{equation}
\begin{equation}
g^{(p)}:= dim \Gamma(M, \Omega^{p}_M)/\pi^{*}\Gamma(V, \Omega^p_V).
\end{equation}
The $s$-invariant of the singularity is defined as follows:
\begin{equation}
s:= dim \Gamma(M\backslash A, \Omega^{n})/[\Gamma(M, \Omega^{n})+d
\Gamma(M\backslash A, \Omega^{n-1})].
\end{equation}}
\end{definition}

\begin{lemma}([Lu-Ya2])\label{pg q}
\emph{Let $V$ be a $n$-dimensional Stein space with $0$ as its only
singular point. Let $\pi: (M, A)\rightarrow (V, 0)$ be a resolution
of the singularity with $A$ as exceptional set. Let $J^{\bullet}$ be
the Siu complex of coherent sheaves supported on $0$. Then:
\begin{enumerate}
\item[1 .]
$dim J^n=p_g$.
\item[2 .]
$dim J^{n-1}=q$.
\item[3 .]
$dim J^{i}=0$, for $1\le i\le n-2$.
\end{enumerate}}
\end{lemma}

\begin{proposition}([Lu-Ya2])
\emph{Let $V$ be a $n$-dimensional Stein space with $0$ as its only
singular point. Let $\pi: (M, A)\rightarrow (V, 0)$ be a resolution
of the singularity with $A$ as exceptional set. Let $J^{\bullet}$ be
the Siu complex of coherent sheaves supported on $0$. Then the
$s$-invariant is given by
\begin{equation}
s:=dim H^n(J^{\bullet})=p_g-q
\end{equation}
and
\begin{equation}
dim H^{n-1}(J^{\bullet})=0.
\end{equation}}
\end{proposition}

Let $X$ be a compact connected strongly pseudoconvex $CR$ manifold
of real dimension $3$, in the boundary of a bounded strongly
pseudoconvex domain $D$ in $\mathbb{C}^N$. By Harvey and Lawson
[Ha-La], there is a unique complex variety $V$ in $\mathbb{C}^N$
such that the boundary of $V$ is $X$. Let $\pi: (M, A_1,\cdots,
A_k)\rightarrow (V, 0_1,\cdots, 0_k)$ be a resolution of the
singularities with $A_i=\pi^{-1}(0_i)$, $1\le i\le k$, as
exceptional sets. Then the $s$-invariant defined in Definition
\ref{s inv} is $CR$ invariant, which is also called $s(X)$.

In order to solve the classical complex Plateau problem, we need to
find some $CR$-invariant which can be calculated directly from the
boundary $X$ and the vanishing of this invariant will give the
regularity of Harvey--Lawson solution to the complex Plateau
problem.

For this purpose, we define a new sheaf $\bar{\bar{\Omega}}_V^{1,
1}$.

\begin{definition}
\emph{Let $(V, 0)$ be a Stein germ of a $2$-dimensional analytic
space with an isolated singularity at $0$. Define a sheaf of germs
$\bar{\bar{\Omega}}_V^{1, 1}$ by the sheaf associated to the
presheaf
\[
U\mapsto <\Gamma(U, \bar{\bar{\Omega}}^1_{V})\wedge \Gamma(U,
\bar{\bar{\Omega}}^1_{V})>,
\]
where $U$ is an open set of $V$.}
\end{definition}

\begin{lemma}\label{loc inv2}
\emph{Let $V$ be a $2$-dimensional Stein space with $0$ as its only
singular point in $\mathbb{C}^N$. Let $\pi: (M, A)\rightarrow (V,
0)$ be a resolution of the singularity with $A$ as exceptional set.
Then $\bar{\bar{\Omega}}_V^{1, 1}$ is coherent and there is a short
exact sequence
\begin{equation}
0\longrightarrow\bar{\bar{\Omega}}_V^{1,
1}\longrightarrow\bar{\bar{\Omega}}_V^2\longrightarrow\mathscr{G}^{(1,1)}\longrightarrow
0
\end{equation}
where $\mathscr{G}^{(1,1)}$ is a sheaf supported on the singular
point of $V$. Let
\begin{equation}
G^{(1,1)}(M\backslash A):=\Gamma(M\backslash A,
\Omega^2_M)/<\Gamma(M\backslash A, \Omega^1_M)\wedge
\Gamma(M\backslash A, \Omega^1_M)>;
\end{equation}
then $dim \mathscr{G}^{(1,1)}_0=dim G^{(1,1)}(M\backslash A)$.}
\end{lemma}
\begin{proof}
Since the sheaf of germ $\bar{\bar{\Omega}}^1_V$ is coherent by a
theorem of Siu (see Theorem A of [Si]), for any point $w\in V$ there
exists an open neighborhood $U$ of $w$ in $V$ such that $\Gamma(U,
\bar{\bar{\Omega}}^1_{V})$ is finitely generated over $\Gamma(U,
\mathscr{O}_{V})$. So $\Gamma(U, \bar{\bar{\Omega}}^1_{V})\wedge
\Gamma(U, \bar{\bar{\Omega}}^1_{V})$ is finitely generated over
$\Gamma(U, \mathscr{O}_{V})$, which means $\Gamma(U,
\bar{\bar{\Omega}}_{V}^{1, 1})$ is finitely generated over
$\Gamma(U, \mathscr{O}_{V})$ -- i.e., $\bar{\bar{\Omega}}_{V}^{1,
1}$ is a sheaf of finite type. It is obvious that
$\bar{\bar{\Omega}}_{V}^{1, 1}$ is a subsheaf of
$\bar{\bar{\Omega}}_{V}^2$ which is also coherent. So
$\bar{\bar{\Omega}}_{V}^{1, 1}$ is coherent.

Notice that the stalk of $\bar{\bar{\Omega}}_V^{1, 1}$ and
$\bar{\bar{\Omega}}_V^2$ coincide at each point different from the
singular point $0$, so $\mathscr{G}^{(1,1)}$ is supported at $0$.
And from Cartan Theorem B
$$dim
\mathscr{G}^{(1,1)}_0=dim \Gamma(V,
\bar{\bar{\Omega}}_V^2)/\Gamma(V, \bar{\bar{\Omega}}_V^{1, 1}) =dim
G^{(1,1)}(M\backslash A).$$
\end{proof}

Thus, from Lemma \ref{loc inv2}, we can define a local invariant of
a singularity which is independent of resolution.
\begin{definition}
\emph{Let $V$ be a $2$-dimensional Stein space with $0$ as its only
singular point. Let $\pi: (M, A)\rightarrow (V, 0)$ be a resolution
of the singularity with $A$ as exceptional set. Let
\begin{equation}
g^{(1,1)}(0):=dim \mathscr{G}^{(1,1)}_0=dim G^{(1,1)}(M\backslash
A).
\end{equation}}
\end{definition}

We will omit $0$ in $g^{(1,1)}(0)$ if there is no confusion from the
context.

Let $\pi: (M, A_1,\cdots, A_k)\rightarrow (V, 0_1,\cdots, 0_k)$ be a
resolution of the singularities with $A_i=\pi^{-1}(0_i)$, $1\le i\le
k$, as exceptional sets. In this case, we still let
$$G^{(1,1)}(M\backslash A):=\Gamma(M\backslash A, \Omega^2_M)/<\Gamma(M\backslash A,
\Omega^1_M)\wedge \Gamma(M\backslash A, \Omega^1_M)>.$$

\begin{definition}
\emph{If $X$ is a compact connected strongly pseudoconvex $CR$
manifold of real dimension $3$ which is in the boundary of a bounded
strongly pseudoconvex domain $D$ in $\mathbb{C}^N$. Suppose $V$ in
$\mathbb{C}^N$ such that the boundary of $V$ is $X$. Let $\pi: (M,
A=\bigcup_i A_i)\rightarrow (V, 0_1,\cdots, 0_k)$ be a resolution of
the singularities with $A_i=\pi^{-1}(0_i)$, $1\le i\le k$, as
exceptional sets. Let
\begin{equation}
G^{(1,1)}(M\backslash A):=\Gamma(M\backslash A ,
\Omega^2_M)/<\Gamma(M\backslash A, \Omega^1_M)\wedge
\Gamma(M\backslash A, \Omega^1_M)>
\end{equation}
and
\begin{equation}
G^{(1,1)}(X):=\mathscr{S}^2(X)/<\mathscr{S}^1(X)\wedge
\mathscr{S}^1(X)>
\end{equation}
where $\mathscr{S}^p$ are holomorphic cross sections of
$\wedge^p(\widehat{T}(X)^*)$. Then we set
\begin{equation}
g^{(1,1)}(M\backslash A):=dim G^{(1,1)}(M\backslash A),
\end{equation}
\begin{equation}
g^{(1,1)}(X):=dim G^{(1,1)}(X).
\end{equation}}
\end{definition}

\begin{lemma}\label{boundary}
\emph{Let $X$ be a compact connected strongly pseudoconvex $CR$
manifold of real dimension $3$ which bounds a bounded strongly
pseudoconvex variety $V$ with only isolated singularities
$\{0_1,\cdots, 0_k\}$ in $\mathbb{C}^N$. Let $\pi: (M, A_1,\cdots,
A_k)\rightarrow (V, 0_1,\cdots, 0_k)$ be a resolution of the
singularities with $A_i=\pi^{-1}(0_i)$, $1\le i\le k$, as
exceptional sets. Then $g^{(1,1)}(X)=g^{(1,1)}(M\backslash A)$,
where $A=\cup A_i$, $1\le i\le k$.}
\end{lemma}
\begin{proof}
Take a one-convex exhausting function $\phi$ on $M$ such that
$\phi\ge 0$ on $M$ and $\phi(y)=0$ if and only if $y\in A$. Set
$M_r=\{y\in M, \phi(y)\ge r\}$. Since $X=\partial M$ is strictly
pseudoconvex, any holomorphic $p$-form $\theta\in \mathscr{S}^p(X)$
can be extended to a one-sided neighborhood of $X$ in $M$. Hence,
$\theta$ can be thought of as holomorphic $p$-form on $M_r$-- i.e.,
an element in $\Gamma(M_r, \Omega^p_{M_r})$. By Andreotti and
Grauert ([An-Gr]), $\Gamma(M_r, \Omega^p_{M_r})$ is isomorphic to
$\Gamma(M\backslash A, \Omega^p_{M})$. So
$g^{(1,1)}(X)=g^{(1,1)}(M\backslash A)$.
\end{proof}

 By Lemma \ref{boundary} and the proof of Lemma \ref{loc inv2}, we can get the following lemma easily.

\begin{lemma}\label{g11}
\emph{Let $X$ be a compact connected strongly pseudoconvex $CR$
manifold of real dimension $3$, which bounds a bounded strongly
pseudoconvex variety $V$ with only isolated singularities
$\{0_1,\cdots, 0_k\}$ in $\mathbb{C}^N$. Then $g^{(1,1)}(X)=\sum_i
g^{(1,1)}(0_i)=\sum_i dim \mathscr{G}^{(1,1)}_{0_i}$.}
\end{lemma}

The following proposition is to show that $g^{(1,1)}$ is bounded
above.
\begin{proposition}
\emph{Let $V$ be a $2$-dimensional Stein space with $0$ as its only
singular point. Then $g^{(1,1)}\le p_g+g^{(2)}$.}
\end{proposition}
\begin{proof}
Since $$g^{(1,1)}=dim\Gamma(M\backslash A,
\Omega^2_M)/<\Gamma(M\backslash A, \Omega^1_M)\wedge
\Gamma(M\backslash A, \Omega^1_M)>,$$
\[p_g=dim \Gamma(M\backslash A,
\Omega^2_M)/\Gamma(M, \Omega^2_M),
\]
\[
g^{(2)}:= dim \Gamma(M, \Omega^{2})/\pi^{*}\Gamma(V, \Omega^2_V),
\] and
\begin{equation}
\begin{split}
\pi^{*}\Gamma(V, \Omega^2_V)&=<\pi^{*}\Gamma(V, \Omega^1_V)\wedge
\pi^{*}\Gamma(V, \Omega^1_V)>\\
&\subseteq \Gamma(M, \Omega^1_M)\wedge \Gamma(M, \Omega^1_M)\\
&\subseteq \Gamma(M\backslash A, \Omega^1_M)\wedge\Gamma(M\backslash
A, \Omega^1_M),
\end{split}
 \end{equation} the result follows.
\end{proof}

The following theorem is the crucial part for the classical complex
Plateau problem.

\begin{theorem}\label{new inv}
\emph{Let $V$ be a $2$-dimensional Stein space with $0$ as its only
normal singular point with $\mathbb{C}^*$-action. Let $\pi: (M,
A)\rightarrow (V, 0)$ be a minimal good resolution of the
singularity with $A$ as exceptional set, then $g^{(1,1)}\ge 1$.}
\end{theorem}
\begin{proof}
If $dim \Gamma(M\backslash A, \Omega_M^2)/\Gamma(M, \Omega_M^2)>0$,
then there exists
$$\omega_0\in \Gamma(M\backslash A,
\Omega_M^2)\backslash\Gamma(M, \Omega_M^2).$$ So $\omega_0$ must
have pole along some irreducible component $A_k$ of $A$. Suppose
$\omega$ has the highest order of pole along $A_k$ and $\omega\in
\Gamma(M\backslash A, \Omega_M^2)$. Denote $Ord_{A_k}\omega=r<0$.
Let $z_1,\cdots, z_m$ be coordinate functions of $\mathbb{C}^m$.
Choose a point $b$ in $A_k$ which is a smooth point of $A$.
Let$(x_1, x_2)$ be a coordinate system centered at $b$ such that
$A_k$ is given locally by $x_1=0$ at $b$. Take the power series
expansion of $\pi^*(z_j)$ around $b$:
\begin{equation}
\pi^*(z_j)=x_1^{r_j}f_j, 1\le j\le m,
\end{equation}
where $f_j$ is holomorphic function such that $f_j(0, x_2)\neq 0$.
So by the choice of $\omega$, $min\{r_1, \dots, r_m\}>0>r$.

Let $\xi_V\in \Gamma(V, \Theta_V)$, where
$\Theta_V:={\mathscr{H}om}_{\mathscr{O}_V}(\Omega^1_V,
\mathscr{O}_V)$, denote the generating vector field of the
$\mathbb{C}^*$-action and $i_{\xi_V}$ be the contraction map. For
some $\alpha\in \Gamma(V, \bar{\bar{\Omega}}^1_V)$, write $\alpha$
as a sum $\sum\alpha^j$ of quasi-homogeneous elements where
$\alpha^j$ is a quasi-homogeous element of degree $l_j>0$.  Let
$L_{\xi_V}=i_{\xi_V}d+di_{\xi_V}$ be the Lie derivation. Then
\[l_j\alpha^j=L_{\xi_V}\alpha^j=i_{\xi_V}d(\alpha^j)+di_{\xi_V}(\alpha^j).
\]
So
\begin{equation}\label{seq2}
\Gamma(V, \bar{\bar{\Omega}}^1_V)=d(\Gamma(V, \mathscr{O}_V))+
i_{\xi_V}(\Gamma(V, \bar{\bar{\Omega}}^2_V)).
\end{equation}

For minimal good resolution, we have $\pi_*\Theta_M=\Theta_V$ (see
[Bu-Wa]), where $\Theta_M$ is the vector field on $M$. Thus, there
exists ${\xi_M}$ which is a lift of ${\xi_V}$ -- i.e.,
$\pi_*{\xi_M}={\xi_V}$. We know that ${\xi_M}$ is tangential to the
exceptional set, so
\[
\xi_M\circeq x_1^{a_1}p\frac{\partial}{\partial
x_1}+x_1^{a_2}q\frac{\partial}{\partial x_2}, a_1\ge 1, a_2\ge 0
\]
where $p$ and $q$ are holomorphic functions.

Let $i_{\xi_M}: \Gamma(M\backslash A, \Omega^2_M)\longrightarrow
\Gamma(M\backslash A, \Omega^1_M)$ be the contraction map
corresponding to $i_{\xi_V}$. If $\zeta\in \Gamma(M\backslash A,
\Omega^2_M)$ and $\zeta\circeq x _1^ugdx_1\wedge dx_2$, then
\[
i_{\xi_M}(\zeta)\circeq i_{\xi_M}(x_1^ugdx_1\wedge
dx_2)=-x_1^{u+a_2}qgdx_1+x_1^{u+a_1}pgdx_2.
\]
From (\ref{seq2}),
\[
\Gamma(M\backslash A, \Omega^1_M)=d(\Gamma(M\backslash A,
\mathscr{O}_M))+ i_{\xi_M}(\Gamma(M\backslash A, \Omega^2_M)).
\]
Since $V$ is normal , $g^{(0)}=0$ -- i.e., $\Gamma(M,
\mathscr{O}_M)=\pi^*(\Gamma(V, \mathscr{O}_V))$. Moreover, by the
normality of $(V, 0)$, $\Gamma(M, \mathscr{O}_M)=\Gamma(M\backslash
A, \mathscr{O}_M)$.

We now prove that $\omega$ is not contained in $<\Gamma(M\backslash
A, \Omega^1_M)\wedge\Gamma(M\backslash A, \Omega^1_M)>$. Consider
$\eta, \varphi \in \Gamma(M\backslash A, \Omega^1_M)$ locally around
$b$

Suppose $\eta=\eta_1+\eta_2$ and $\varphi=\varphi_1+\varphi_2$,
where $\eta_1$, $\varphi_1 \in d(\Gamma(M, \mathscr{O}_M))$,
$\eta_2$, $\varphi_2 \in i_{\xi_M}(\Gamma(M\backslash A,
\Omega^2_M))$. Let
$$\eta_2=i_{\xi_M}(\zeta), \hspace{.5cm}\zeta\circeq x _1^ugdx_1\wedge
dx_2, \hspace{.5cm} g(0, x_2)\neq 0$$ and
$$\varphi_2=i_{\xi_M}(\varsigma), \hspace{.5cm}\varsigma\circeq x _1^vhdx_1\wedge
dx_2, \hspace{.5cm} h(0, x_2)\neq 0.$$ So $u$ and $v$ are bounded
lower by $r$.

Then
$$\eta\wedge\varphi=\eta_1\wedge\varphi_1+(\eta_1\wedge\varphi_2+\eta_2\wedge\varphi_1)+\eta_2\wedge\varphi_2.$$

 Since
\[
d\pi^*(z_i)\wedge d\pi^*(z_j)=(r_ix_1^{r_i+r_j-1}f_i\frac{\partial
f_j}{\partial x_2}-r_jx_1^{r_i+r_j-1}f_j\frac{\partial f_i}{\partial
x_2})dx_1\wedge dx_2,
\]
$Ord_{A_k} \eta_1\wedge\varphi_1\ge 2\cdot min\{r_1, \dots,
r_m\}-1>r$.

Write $\eta_2$ and $\varphi_2$ locally around $b$:
$$\eta_2\circeq -x_1^{u+a_2}qgdx_1+x_1^{u+a_1}pgdx_2,$$
$$\varphi_2\circeq -x_1^{v+a_2}qhdx_1+x_1^{v+a_1}phdx_2.$$
So $\eta_2\wedge\varphi_2=\circeq0$.

Also notice that
$$d\pi^*(z_j)=r_jx_1^{r_j-1}f_jdx_1+x_1^{r_j}\frac{\partial
f_j}{\partial x_2}dx_2.$$

So $$Ord_{A_k} \eta_1\wedge \varphi_2\ge min\{r_1, \dots,
r_m\}+v>r$$ and
$$Ord_{A_k} \eta_2\wedge\varphi_1\ge
min\{r_1, \dots, r_m\}+u>r.$$

From the discussion above, we can get $Ord_{A_k} \eta\wedge
\varphi>r.$

Therefore, $\omega$ is not a linear combination of elements in
$<\Gamma(M\backslash A, \Omega^1_M)\wedge\Gamma(M\backslash A,
\Omega^1_M)>$.

\vspace{.5cm} If $dim \Gamma(M\backslash A, \Omega_M^2)/\Gamma(M,
\Omega_M^2)=0$, the singularity is rational. So irregularity $q=0$
(see [Ya4]). Then
$$\frac{\Gamma(M\backslash A , \Omega^2_M)}{<\Gamma(M\backslash A,
\Omega^1_M)\wedge \Gamma(M\backslash A, \Omega^1_M)>}=\frac{\Gamma(M
, \Omega^2_M)}{<\Gamma(M, \Omega^1_M)\wedge\Gamma(M,
\Omega^1_M)>},$$
\[
g^{(1,1)}=dim\frac{\Gamma(M , \Omega^2_M)}{<\Gamma(M,
\Omega^1_M)\wedge \Gamma(M, \Omega^1_M)>}.
\]
From [Ya3], the canonical bundle $K_M$ is generated by its global
sections in a neighborhood of the exceptional set. So there exists
$\omega\in\Gamma(M, \Omega^2_M)$ such that $\omega$ does not vanish
along some irreducible component $A_k$ of $A$. The rest of the
argument is same as those arguments above with $r=0$ -- i.e., we can
get $\omega$ is not a linear combination of elements in $<\Gamma(M,
\Omega^1_M)\wedge\Gamma(M, \Omega^1_M)>$. So
\[
g^{(1,1)}=dim\frac{\Gamma(M , \Omega^2_M)}{<\Gamma(M,
\Omega^1_M)\wedge \Gamma(M, \Omega^1_M)>}\ge 1.
\]

\end{proof}

\section{\textbf{The classical complex Plateau problem}}
In 1981, Yau [Ya] solved the classical complex Plateau problem for
the case $n\ge 3$.
\begin{theorem}([Ya])\label{Yau Pla}
\emph{Let $X$ be a compact connected strongly pseudoconvex $CR$
manifold of real dimension $2n-1$, $n\ge 3$, in the boundary of a
bounded strongly pseudoconvex domain $D$ in $\mathbb{C}^{n+1}$. Then
$X$ is a boundary of the complex sub-manifold $V\subset D-X$ if and
only if Kohn--Rossi cohomology groups $H^{p, q}_{K R}(X)$ are zeros
for $1\le q\le n-2$}
\end{theorem}

Next, we want to use our new invariants introduced in $\S$ 3 to
solve the classical complex Plateau problem for the case $n=2$.

First, we present some known results from the paper [Lu-Ya2].

\begin{theorem}([Lu-Ya2])\label{b to in}
\emph{Let $X$ be a compact connected $(2n-1)$-dimensional $(n \ge2)$
strongly pseudoconvex CR manifold. Suppose that X is the boundary of
an $n$-dimensional strongly pseudoconvex manifold $M$ which is a
modification of a Stein space $V$ with only isolated singularities
$\{0_1,\dots, 0_m\}$. Let $A$ be the maximal compact analytic set in
$M$ which can be blown down to $\{0_1,\dots, 0_m\}$. Then:
\begin{enumerate}
\item[1 .]
$H^q_h(X)\cong H^q_h(M\backslash A)\cong H^q_h(M), \hspace{1cm} 1\le
q\le n-1$.
\item[2 .]
$H^n_h(X)\cong H^n_h(M\backslash A), dim H^n_h(M\backslash A)=dim
H^n_h(M)+s$, where $s=s_1+\cdots+s_m$, $s_i$ is the $s$-invariant of
the singularity $(V, 0_i)$.
\end{enumerate}}
\end{theorem}

\begin{theorem}([Lu-Ya2])\label{s=0}
\emph{Let $(V, 0)$ be a Gorenstein surface singularity. Let $\pi:
M\rightarrow V$ be a good resolution with $A=\pi^{-1}(0)$ as
exceptional set. Assume that $M$ is contractible to $A$. If $s=0$,
then $(V, 0)$ is a quasi-homogeneous singularity,
$H^1(A,\mathbb{C})=0$, $dim H^1(M,\Omega^1)=dim H^2(A,\mathbb{C})
+dim H^1(M,\mathscr{O})$, and $H^1_h(M)=H^2_h(M)=0$. Conversely, if
$(V, 0)$ is a 2-dimensional quasi-homogeneous Gorenstein singularity
and $H^1(A,\mathbb{C})=0$, then the $s$-invariant vanishes.}
\end{theorem}

\begin{theorem}([Lu-Ya2])\label{th h2=0}
\emph{Let $X$ be a strongly pseudoconvex compact Calabi--Yau $CR$
manifold of dimension $3$. Suppose that $X$ is contained in the
boundary of a strongly pseudoconvex bounded domain $D$ in
$\mathbb{C}^N$. If the holomorphic De Rham cohomology $H^2_h(X)=0$,
then $X$ is a boundary of a complex variety $V$ in $D$ with boundary
regularity and $V$ has only isolated singularities in the interior
and the normalizations of these singularities are Gorenstein surface
singularities with vanishing $s$-invariant.}
\end{theorem}

\begin{corollary}([Lu-Ya2])\label{h2=0}
\emph{Let $X$ be a strongly pseudoconvex compact $CR$ manifold of
dimension $3$. Suppose that $X$ is contained in the boundary of a
strongly pseudoconvex bounded domain D in $\mathbb{C}^3$. If the
holomorphic De Rham cohomology $H^2_h(X)=0$, then $X$ is a boundary
of a complex variety $V$ in $D$ with boundary regularity and $V$ has
only isolated quasi-homogeneous singularities such that the dual
graphs of the exceptional sets in the resolution are star shaped and
all the curves are rational.}
\end{corollary}

So from several theorems above we can see that, in the paper
[Lu-Ya2], Luk and Yau give a sufficient condition  $H^2_h(X)=0$ to
determine when $X$ can bound some special singularities. However,
even if both $H^1_h(X)$ and $H^2_h(X)$ vanish, $V$ still can be
singular.

We use $CR$ invariants given in the last section to get sufficient
and necessary conditions for the variety bounded by $X$ being smooth
after normalization.

\begin{theorem}
\emph{Let $X$ be a strongly pseudoconvex compact Calabi--Yau $CR$
manifold of dimension $3$. Suppose that $X$ is contained in the
boundary of a strongly pseudoconvex bounded domain $D$ in
$\mathbb{C}^N$. Then $X$ is a boundary of the complex variety
$V\subset D-X$ with boundary regularity and the variety is smooth
after normalization if and only if $s$-invariant and $g^{(1,1)}(X)$
vanish.}
\end{theorem}
\begin{proof}
$(\Rightarrow)$ : Since $V$ is smooth after normalization,
$g^{(1,1)}(X)=0$ follows from Lemma \ref{g11}.

$(\Leftarrow)$ : It is well known that $X$ is a boundary of a
variety $V$ in $D$ with boundary regularity ([Lu-Ya, Ha-La2]). Since
$s=0$, $X$ is a boundary of the complex sub-manifold $V\subset D-X$
with only isolated Gorenstein quasi-homogeneous singularities
$\{0_1,\cdots, 0_k \}$ after normalization. Let $\pi_i:
M_i\rightarrow V_i$ be the minimal good resolution of a sufficiently
small neighborhood $V_i$ of $0_i$ in $V$, $1\le i\le k$. From
Theorem \ref{new inv}, $dim G^{(1,1)}(M_i)>0$, which contradicts
$g^{(1,1)}(X)=0$. So $V$ is smooth.
\end{proof}

\begin{corollary}
\emph{Let $X$ be a strongly pseudoconvex compact $CR$ manifold of
dimension $3$. Suppose that $X$ is contained in the boundary of a
strongly pseudoconvex bounded domain $D$ in $\mathbb{C}^3$. Then $X$
is a boundary of the complex sub-manifold $V\subset D-X$ if and only
if $s$-invariant and $g^{(1,1)}(X)$ vanish.}
\end{corollary}

From Theorem \ref{b to in}, we know that if $H_h^2{(X)}=0$, and then
$s=0$. So we can get a necessary and sufficient condition in terms
of boundary $X$, with $H_h^2{(X)}=0$, to determine when $X$ is a
boundary of a manifold up to normalization.

\begin{corollary}
\emph{Let $X$ be a strongly pseudoconvex compact Calabi--Yau $CR$
manifold of dimension $3$. Suppose that $X$ is contained in the
boundary of a strongly pseudoconvex bounded domain $D$ in
$\mathbb{C}^N$ with $H_h^2{(X)}=0$. Then $X$ is a boundary of the
complex sub-manifold up to normalization $V\subset D-X$ with
boundary regularity if and only if $g^{(1,1)}(X)=0$.}
\end{corollary}

\begin{corollary}
\emph{Let $X$ be a strongly pseudoconvex compact $CR$ manifold of
dimension $3$. Suppose that $X$ is contained in the boundary of a
strongly pseudoconvex bounded domain $D$ in $\mathbb{C}^3$ with
$H_h^2{(X)}=0$. Then $X$ is a boundary of the complex sub-manifold
$V\subset D-X$ if and only if $g^{(1,1)}(X)=0$ .}
\end{corollary}


\end{document}